\documentclass[a4paper,10pt]{amsart}

\usepackage{amsmath}
\usepackage{amsfonts}
\usepackage{amssymb}
\usepackage{graphicx}
\usepackage{color}
\usepackage{soul,xcolor} 
\usepackage{cite}

\usepackage{amsthm}
\usepackage[all,cmtip]{xy}

\newcommand{\mbN}{\mathbb{N}}
\newcommand{\mbQ}{\mathbb{Q}}

\newcommand{\mcO}{\mathcal{O}}

\DeclareMathOperator{\ch}{ch}
\DeclareMathOperator{\td}{td}

\DeclareMathOperator{\Exc}{Exc}

\newcommand*{\coloneq}{\mathrel{\mathop:}=}

\theoremstyle{plain}
\newtheorem{theorem}{Theorem}[section]
\newtheorem{proposition}[theorem]{Proposition}
\newtheorem{lemma}[theorem]{Lemma}
\newtheorem{corollary}[theorem]{Corollary}

\theoremstyle{definition}

\theoremstyle{remark}
\newtheorem{remark}[theorem]{Remark}

\usepackage{hyperref}

\hypersetup{
	colorlinks=true,         
	linkcolor=blue,          
	citecolor=blue,          
	hyperfootnotes=false,
}

\begin{document}
\bibliographystyle{alpha}

\title[Rational curves on fibered Calabi--Yau manifolds]{Rational curves \\ on fibered Calabi--Yau manifolds}

\author{Simone Diverio}
\address{Simone Diverio \\  Dipartimento di Matematica \lq\lq Guido Castelnuovo\rq\rq{},
SAPIENZA Universit\`a di Roma,
Piazzale Aldo Moro 5,
00185 Roma, Italy.}
\email{diverio@mat.uniroma1.it} 

\author{Claudio Fontanari}
\address{Claudio Fontanari \\ Universit\`a degli Studi di Trento, Dipartimento di Matematica, via Sommarive 14, 38123 Povo, Trento, Italy.}
\email{claudio.fontanari@unitn.it} 

\author{Diletta Martinelli}
\address{Diletta Martinelli \\ School of Mathematics, University of Edinburgh, Peter Guthrie
Tait Road, Edinburgh SW7 2AZ, UK.}
\email{Diletta.Martinelli@ed.ac.uk}

\thanks{The first--named author was partially supported by the ANR Programme: D\'efi de tous les savoirs (DS10) 2015, \lq\lq GRACK\rq\rq, Project ID: ANR-15-CE40-0003ANR, and D\'efi de tous les savoirs (DS10) 2016, \lq\lq FOLIAGE\rq\rq, Project ID: ANR-16-CE40-0008. The second--named author was partially supported by GNSAGA of INdAM and FIRB 2012 "Moduli spaces and Applications". The third--named author was partially supported by the Stevenson Fund and by FIRB 2012 "Moduli spaces and Applications".}

\begin{abstract} 
We show that a smooth projective complex manifold of dimension greater than two endowed with an elliptic fiber space structure and with finite fundamental group always contains a rational curve, provided its canonical bundle is relatively trivial.	
As an application of this result, we prove that any Calabi--Yau manifold that admits a fibration onto a curve whose general fiber is an abelian variety always contains a rational curve.
\end{abstract}

\date{}

\maketitle

\section{Introduction}

The first goal of this paper is to prove the following result about the existence of rational curves on an elliptically fibered projective manifold $X$ with some restriction on its fundamental group. We always assume that the dimension of $X$ is greater than two.

\begin{theorem}\label{thm:main}
	Let $X$ be a smooth projective manifold with finite fundamental group. Suppose there exists a projective variety $B$ and a morphism $f\colon X\to B$ such that the general fiber has dimension one. Suppose, moreover, that there exists a line bundle $L$ on $B$ such that $K_X\simeq f^* L$. Then $X$ does contain a rational curve. 
\end{theorem}

An \emph{elliptic fiber space} is a projective variety endowed with a fibration of relative dimension one such that its general fiber is an elliptic curve. The assumption on the canonical bundle in the statement readily implies, by adjunction, that the manifold $X$ in the statement is indeed an elliptic fiber space, whenever with connected fibers.

By a \emph{Calabi--Yau manifold} we mean a smooth projective manifold $X$ with trivial canonical bundle $K_X\simeq\mathcal O_X$ and finite fundamental group. Calabi--Yau manifolds are of interest in both algebraic geometry and theoretical physics. In particular, the problem of determining whether they do contain rational curves is important in string theory (see for instance \cite{Wit86,DSWW86}, where the physical relevance of rational curves on Calabi--Yau manifolds is discussed). Moreover, a \emph{folklore} conjecture in algebraic geometry predicts the existence of rational curves on every Calabi--Yau 
manifold, see for instance \cite[Question 1.6]{Miy94}, and \cite[Problem 10.2]{MP97}. Already in dimension three, the conjecture is open. There are results for high Picard rank (see \cite{Wil89} and 
\cite{HBW92}), in the case of existence of a non-zero, effective, non-ample divisor (see \cite{Pet91} 
and \cite{Ogu93}), and in the case of existence of a non-zero, nef, non-ample divisor (see \cite{DF14}). 
Almost nothing is known in higher dimension.

Here, we consider the case of an \emph{elliptic Calabi--Yau manifold}, i.e. a Calabi--Yau manifold which is also an elliptic fiber space. Quoting Koll\'ar \cite{Kol15}, \lq\lq $F$-theory posits that the hidden dimensions constitute a Calabi--Yau $4$-fold $X$ that has an elliptic structure with a section\rq\rq{}.

As an immediate corollary of Theorem \ref{thm:main}, we obtain that on elliptic Calabi--Yau's, there is always at least one rational curve. It is a generalization of \cite[Theorem 3.1, Case $\nu(X,L)=2$]{Ogu93}, who treated therein the three dimensional case.

\begin{corollary} \label{cor:ellCY}
	Let $X$ be an elliptic Calabi--Yau manifold. Then $X$ always contains a rational curve.
\end{corollary}

\begin{remark}
	Calabi--Yau's in dimension two are just $K3$ surfaces, that are known to contain rational curves thanks to the Bogomolov--Mumford Theorem, \cite{MM83}.
\end{remark}

Conjecturally, a Calabi--Yau manifold $X$ is elliptic if and only if there exists a $(1,1)$-class $\alpha\in H^2(X,\mathbb Q)$ such that $\alpha$ is nef of numerical dimension $\nu(\alpha)=\dim X-1$ (this enters in the circle of ideas around the Generalized Abundance Conjecture for Calabi--Yau manifolds). Recall that the numerical dimension of a nef class $\alpha$ is the biggest integer $m$ such that self-intersection $\alpha^{m}$ is non zero in $H^{2m}(X,\mathbb Q)$. This conjecture is known to hold true, under the further assumption that $\alpha^{\dim X-2}\cdot c_2(X)\ne 0$, for threefolds by the work of \cite{Wil94,Ogu93} and in all dimensions by \cite[Corollary 11]{Kol15}. We can thus state the following numerical sufficient criterion for the existence of rational curves on Calabi--Yau manifolds.

\begin{corollary}\label{cor:num}
	Let $X$ be a Calabi--Yau manifold. Suppose that $X$ possesses a $(1,1)$-class $\alpha\in H^2(X,\mathbb Q)$ such that 
	\begin{itemize}
		\item the class $\alpha$ is nef,
		\item the numerical dimension $\nu(\alpha)$ of $\alpha$ is $\dim X-1$,
		\item the intersection product $\alpha^{\dim X-2}\cdot c_2(X)$ is non zero.
	\end{itemize}
	Then $X$ is elliptic and therefore it contains a rational curve.
\end{corollary}

In order to discuss another corollary, let us recall the following. Suppose that $X$ is a projective manifold with semi-ample canonical bundle, \textsl{i.e.} some tensor power of $K_X$ is globally generated. Then there exists on $X$ an algebraic fiber space structure (see Subsection \ref{sec:notconv} for the definition) $\phi\colon X\to B$, called the semi-ample Iitaka fibration. This algebraic fiber space has the property that $\dim B=\kappa(X)$ and that there exists an ample line bundle $A$ over $B$ such that $K_X^{\otimes\ell}\simeq\phi^*A$ (see \cite[Theorem 2.1.27]{Laz04}). Here, $\ell$ is the \emph{exponent} of the sub-semigroup of natural numbers $m$ such that $K_X^{\otimes m}$ is globally generated. In particular, if every sufficiently large power of $K_X$ is free, then $m=1$. 

\begin{corollary}
Let $X$ be a smooth projective manifold with finite fundamental group and Kodaira dimension $\kappa(X)=\dim X-1$. Suppose that $K_X$ is semi-ample of exponent $\ell=1$. Then $X$ contains a rational curve.
\end{corollary}

Observe that the hypothesis of semi-ampleness of $K_X$ can actually be relaxed into nefness. This is because the numerical dimension of a nef line bundle is always greater than or equal to its Kodaira--Iitaka dimension, so that $\nu(K_X)\ge\dim X-1$. Now, if $\nu(K_X)=\dim X$ then $K_X$ would be big and thus $X$ of general type, contradicting $\kappa(X)=\dim X-1$. Thus, we necessarily have $\nu(X)=\kappa(X)=\dim X-1$, and so $K_X$ is semi-ample by \cite[Theorem 1.1]{Kaw85}. On the other hand, the hypothesis on the exponent $\ell=1$ of $K_X$ cannot be dropped in our proof (and, in general, it is not automatic).


The second goal of the paper is to prove an application of Theorem \ref{thm:main}, where we deal with Calabi--Yau manifolds endowed with a fibration onto a curve whose fibers are abelian varieties.

\begin{theorem} \label{thm:applic}
Let $X$ be a Calabi--Yau manifold that admits a fibration $\pi\colon X\to C$ onto a curve whose general fibers are abelian varieties. Then $X$ does contain a rational curve.
\end{theorem}

For explicit examples of fibrations as in Theorem \ref{thm:applic} we refer to \cite[Theorem 4.9]{Ogu93} and to \cite{GP01}. 

A nice feature of the proof of the above theorem is that the rational curves we find are not in some degenerate fiber, but rather tend to be transversal to the original fibration. This fact has a couple of interesting corollaries.

\begin{corollary}\label{cor:multsect}
If $X$ is as in Theorem \ref{thm:applic}, then $\pi\colon X\to C$ has a rational multi-section, \textsl{i.e.} there exists a rational irreducible curve $R\subset X$ such that $\pi|_R\colon  R\to C$ is finite surjective.
\end{corollary}

To state the second consequence, recall that the Kobayashi pseudodistance $d_X$ on a complex space $X$ can be defined as the largest pseudodistance $\delta_X$ on $X$ such that, for any holomorphic map $h \colon \Delta \to X$ from the complex unit disc to $X$, we have $\delta_X(h(p),h(q)) \leq \rho(p,q)$, where $\rho$ is the Poincar{\'e} distance on $\Delta$ (we refer to \cite{Kob98} for more details on this subject). One basic feature of the Kobayashi pseudodistance is that not only holomorphic maps from the Poincar\'e disc to a complex space are distance decreasing, but also general holomorphic maps from any complex space, namely if $f\colon Y\to X$ is any holomorphic map between complex spaces then for all $p,q\in Y$ we have $d_Y(p,q)\ge d_X(f(p),f(q))$. In particular, since $d_Y\equiv 0$ for $Y$ the complex plane, $Y$ a rational curve, or $Y$ a complex torus, if two points of a complex space can be joined by (a chain of) such a $Y$, then these two points have Kobayashi distance zero. Our results, together with this last observation, imply the following.

\begin{corollary}\label{cor:kob}
If $X$ is as in Theorem \ref{thm:applic}, then the Kobayashi pseudodistance $d_X$ on $X$ vanishes identically.
\end{corollary}

This proves a (very) particular case of one of the Campana conjectures on special varieties, see \cite[Conjecture 9.2]{Cam04}.

\subsubsection*{Acknowledgements}

This project started during the visit of the third--named author at the Institut de Math\'ematiques de Jussieu--Paris Rive Gauche and continued during her visit at the Department of Mathematics at the University of Trento. She wishes to thank both institutes for the warm hospitality. 

The first--named author thanks A. H\"oring, A.-S. Kaloghiros and A. Rapagnetta for useful conversations along the years around the subject presented here.

The second--named and the third--named authors would like to thank M. Andreatta, E. Ballico, C. Ciliberto, E. Floris, V. Lazi\'c and A. Petracci for important comments and discussions.

All authors thank R. Svaldi for nice conversations on the subject of the present paper.

\subsection{Notation and conventions}\label{sec:notconv}

We work over the field of complex numbers $\mathbb C$. An \emph{algebraic fiber space} is a surjective proper mapping $f\colon X\to Y$ of projective varieties such that $f_*\mathcal O_X\simeq\mathcal O_Y$; in particular it has connected fibers and if $X$ is normal, so is $Y$. Given a holomorphic (proper) surjective map $f\colon X\to Y$ of smooth complex manifolds, we say that $y\in Y$ is a \emph{regular value} for $f$ if for all $x\in f^{-1}(y)$ the differential $df(x)\colon T_{X,x}\to T_{Y,y}$ is surjective; the set of \emph{singular values} for $f$, \textsl{i.e.} the complement of the set of regular values for $f$, is a proper closed analytic subset of $Y$.

\section{Proof of Theorem \ref{thm:main}}

In this section we shall prove the following result, which will readily imply Theorem \ref{thm:main}.

\begin{theorem}\label{thm:core}
	Let $X$ be a smooth projective manifold. Suppose that $X$ is simply connected and that it is endowed with an elliptic fiber space structure $\phi\colon X\to B$. Suppose moreover that there exist a line bundle $L$ on $B$ such that $K_X\simeq \phi^* L$. Then $X$ does contain a rational curve.
\end{theorem}

Let us first show how Theorem \ref{thm:core} implies Theorem \ref{thm:main}. 

\begin{proof}[Proof of Theorem \ref{thm:main}]
	Let $f\colon X\to B$ as in Theorem \ref{thm:main}. The universal cover $\pi\colon\tilde X\to X$ of $X$ is smooth and projective. Next, consider the Stein factorization of $f\circ\pi$
	$$
	\xymatrix{
		\tilde X \ar[r]^{\phi} \ar@/_/[rr]_{f\circ\pi}& B' \ar[r]^{\nu}  & B,
	}
	$$
	so that $\phi$ has connected fibers, and $B'$ is a normal projective variety. Since $K_{\tilde X}\simeq \pi^* K_X$, we obtain that $K_{\tilde X}\simeq(f\circ\pi)^*L=(\nu\circ\phi)^*L=\phi^* L'$, where $L'$ is the line bundle on $B'$ given by $\nu^* L$. Moreover, $\phi\colon \tilde X\to B'$ is an elliptic fiber space since by construction it is a fiber space whose general fiber has dimension one and, moreover, has trivial canonical bundle. Indeed, let $B^0\subset B'$ the non-empty Zariski open set of regular point of $B'$ which are also regular values for $\phi$, so that $\phi^0=\phi|_{\phi^{-1}(B^0)}$ is a proper holomorphic submersion. Then over $\tilde X^0=\phi^{-1}(B^0)$, the relative tangent bundle sequence 
	\begin{equation} \label{eq:rel_tangent}
	0\to T_{\tilde X^0/B^0}\to T_{\tilde X^0}\to (\phi^0)^*T_{B^0}\to 0
	\end{equation}
	is a short exact sequence of vector bundles. Restricting to one fiber $E$ and taking the determinant of the dual sequence gives a (non canonical) isomorphism
	$$
	K_E\simeq K_{\tilde X^0}|_{E}=K_{\tilde X}|_{E}\simeq \phi^*L|_{E}\simeq\mathcal O_E,
	$$
	and thus $E$ is an elliptic curve.
	Therefore, Theorem \ref{thm:core} applies to $\phi\colon\tilde X\to B'$ and we deduce that $\tilde X$ contains a rational curve $\tilde R\subset\tilde X$. But then $R=\pi(\tilde R)\subset X$ is a rational curve in $X$. 
\end{proof}

We now start the proof of Theorem \ref{thm:core}. We first observe that thanks to the following result of Kawamata, which we state in a slightly simplified version, we can suppose that every fiber of $\phi$ is one dimensional.

\begin{theorem}[Kawamata {\cite[Theorem 2]{Kaw91}}] \label{thm:kaw}
	Let $f\colon X\to Y$ be a surjective projective morphism, where $X$ is smooth and $-K_X$ is $f$-nef (that is, it intersects non negatively the curves which are contracted by $f$). Then any irreducible component of $\operatorname{Exc}(f)=\{x\in X\mid \dim  f^{-1}(f(x))>\dim X-\dim Y\}$ is uniruled.
\end{theorem}

Notice that, if the exceptional set $\operatorname{Exc}(\phi)$ is not empty, then we obtain at once infinitely many rational curves. Moreover, as it follows from the original proof of \cite[Theorem 2]{Kaw91}, the rational ruling curves found are contracted by $f$ (this latter property will be useful later, during the proof of Corollary \ref{cor:multsect}).

Next, we look at the proper subvariety $Z\subset B$ consisting of all the singular points of $B$ and all the singular values of $\phi$. Let $B^0$ be the complement of $Z$ in $B$ and let $X^0$ be its preimage $\phi^{-1}(B^0)$. Thus, the restriction $\phi^0=\phi|_{X^0}\colon X^0\to B^0$ is a proper surjective submersion. 

\begin{lemma}
	The subvariety $Z$ has at least one irreducible component of codimension one in $B$.
\end{lemma}

\begin{proof}
	Suppose the contrary. Then by equidimensionality of the fibers, the complement of $X^0$ in $X$ has codimension at least two. In particular, $\pi_1(X^0)\simeq\pi_1(X)=\{1\}$. Since $\phi^0$ is a proper holomorphic surjective submersion with connected fibers, by Eheresmann's theorem it is a differentiable fiber bundle and thus a Serre fibration. In particular, the long exact sequence for the homotopy groups of the fibration tells us that $B^0$ is simply connected. Now, by \cite{Del68}, in this situation the Leray spectral sequence
	$$
	E^{p,q}_2=H^p(B^0,R^q\phi^0_*\underline{\mathbb Q})\Rightarrow H^{p+q}(X^0,\mathbb Q)
	$$  
	degenerates at $E_2$. Since $B^0$ is simply connected, the locally constant sheaves $R^q\phi^0_*\underline{\mathbb Q}$ are indeed constant, isomorphic to $H^q(E,\mathbb Q)$ where $E$ is (the diffeomorphism class of) a fiber, \textsl{i.e.} a one dimensional complex torus. Now, take $p+q=1$ to get that the graded module associated to some filtration of $H^1(X^0,\mathbb Q)\simeq\{0\}$ has the non zero factor $H^0(B^0,\mathbb Q^2)\simeq \mathbb Q^2$. This is a contradiction.
\end{proof}

\begin{remark}
	Using a result of \cite{KL09} it is possible to conclude as well in a stronger form, as follows. Consider the holomorphic function $J\colon B^0\to\mathbb C$ given by the $j$-invariant of the (elliptic) fibers. Since we suppose that $B^0$ has at least codimension two complement in $B$ and $B$ is normal, $J$ extends to a holomorphic function $B\to\mathbb C$, which must be constant. Thus, all fibers over $B^0$ are isomorphic and by the Grauert--Fischer theorem the family $\phi^0\colon X^0\to B^0$ is locally holomorphically trivial. We are thus in position to apply \cite[Lemma 17]{KL09} which gives, since $B^0$ is moreover simply connected, that $\phi^0\colon X^0\to B^0$ is globally holomorphically trivial. In particular $X^0\simeq E\times B^0$ and thus $X^0$ cannot be simply connected.
\end{remark}

Following \cite{Ogu93}, we shall reduce our situation to the surface case by picking a general curve in $B$ and use Kodaira's canonical bundle formula to study singular fibers. So, let $n = \dim X$ and fix a very ample line bundle $H$ on $B$, positive enough in order to ensure that $H^{\otimes(n-2)}\otimes\mathcal O_B(L)$ is generated by global sections. Observe that, since $\phi\colon X\to B$ is an algebraic fiber space, then $H^0(X,\phi^*H)\simeq H^0(B,\phi_*(\phi^*H))\simeq H^0(B,H\otimes\phi_*\mathcal O_X)\simeq H^0(B,H)$. In particular, general elements in the linear system $|H|$ are also general members of $|\phi^*H|$. Now, take a curve $C\subset B$ which is a general complete intersection of divisors in $|H|$, and the surface $S\subset X$ cut out by the pull-back of such divisors to $X$. By Bertini's theorem, $C$ is normal, hence smooth, and $S$ is smooth, too. Let us still call, by abuse of notation, $\phi\colon S\to C$ the restriction of $\phi$ to $S$. The surface $S$ is an elliptic surface, which we can suppose to be relatively minimal (otherwise we would have found a rational curve on $S$ and hence on $X$).

Since $C$ is general, and the singular locus $B_{\textrm{sing}}$ of $B$ is of codimension two, we can suppose that $C\cap B_{\text{sing}}=\emptyset$. Next, pick a divisorial irreducible component $Z_0$ of $Z$, which always exists thanks to the above lemma. Then $C$ must necessarily intersect $Z_0$, since $C\cdot Z_0=H^{n-2}\cdot Z_0>0$. This means that $\phi\colon S\to C$ must always have at least one singular fiber. Our goal is now to show that such a fiber can never be a multiple fiber and that a singular (non multiple) fiber must necessarily contain an irreducible component which is rational.

Let us start with the following.

\begin{lemma}
	The canonical bundle $K_S$ of $S$ is globally generated.
\end{lemma}

\begin{proof}
	Let $H_1,\dots,H_{n-2}\in|\phi^* H|$ be the smooth divisors in general position which cut out $S$. Then by iterating the adjunction formula we find
	\begin{equation}\label{eq:canS}
	K_S\simeq\bigl(K_X\otimes\phi^*H^{\otimes(n-2)}\bigr)|_{S}\simeq \phi^*\bigl(\mathcal O_B(L)\otimes H^{\otimes(n-2)}\bigr)|_S.
	\end{equation}
	By our choice of $H$, the canonical bundle $K_S$ is the restriction of a pull-back of a globally generated line bundle, hence it is globally generated too.
\end{proof}

Now, recall the (weaker form of the) canonical bundle formula for relatively minimal elliptic fibrations such that its multiple fibers are $S_{c_1}=m_1F_1,\dots,S_{c_k}=m_kF_k$, which reads (see \cite[Corollary V.12.3]{BHPV04}):
\begin{equation}\label{eq:cbf}
K_S\simeq \phi^*G\otimes\mathcal O_S\biggl(\sum_{i=1}^k(m_i-1)\,F_i\biggr),
\end{equation}
where $G$ is some line bundle living over $C$.

\begin{proposition}
	The elliptic surface $S$ does not have any multiple fiber.
\end{proposition}

\begin{proof}
	Indeed, on the one hand the restriction to $K_S$ to any subscheme of $S$ is globally generated, since $K_S$ is globally generated itself. On the other hand, the canonical bundle formula (\ref{eq:cbf}) together with \cite[Lemma III.8.3]{BHPV04} tell us that the restriction $\mathcal O_{F_i}\bigl((m_i-1)F_i\bigr)$ of $K_S$ to $F_{i}$ would have no sections, since $\mathcal O_{F_i}(F_i)$ is torsion of order $m_i$.
\end{proof}

To conclude the proof, we now have to examine singular but not multiple fibers, following Kodaira's table \cite[Section V.7]{BHPV04}. So, let $F$ be such a fiber.

\begin{itemize}
	\item[(i)] If $F$ is irreducible, then it is necessarily rational with a node, or rational with a cusp. In both cases we find a (singular) rational curve on $S$, hence on $X$.
	\item[(ii)] If $F$ is reducible but not multiple, then it is of the form $F=\sum m_iF_i$, with $F_i$ irreducible and reduced, and $\gcd\{m_i\}=1$. In this case, each irreducible component is a smooth rational $(-2)$-curve.
\end{itemize}

\section{Proof of Theorem \ref{thm:applic}}
In this section, after collecting all the essential ingredients, we prove Theorem \ref{thm:applic}, namely, the existence of rational curves on Calabi--Yau manifolds admitting a fibration onto 
a curve whose general fibers are abelian varieties. 
Two standard tools to produce rational curves on a smooth projective variety are the uniruledness of the exceptional loci (see Kawamata's Theorem \ref{thm:kaw} above) and the logarithmic version of the Cone Theorem. We combine them in the following lemma (vaguely inspired by the key lemma in \cite{Wil89}).

\begin{lemma} \label{lem:key}
	Let $X$ be a Calabi--Yau manifold of dimension $n$ such that there exists $D$ a non-ample divisor on $X$ satisfying $D^n > 0$. Assume that $h^i(X, mD) = 0$ for $i > 1$ and for $m$ large enough. Then $X$ does contain a rational curve.
\end{lemma}

\begin{proof}
	We first observe that we may assume $D$ non-nef: otherwise $D$ nef and $D^n > 0$  implies that $D$ is big (see for instance \cite[Theorem 2.2.16]{Laz04}). Then $D$ is semiample by the 
	Basepoint-free Theorem (see for instance \cite[Theorem 3.3]{KM98}) and there exists a multiple of $D$ that defines a surjective generically 1-1 morphism $g \colon X \to Y$. Since $D$ is non-ample, the exceptional locus of $g$ is non-empty, so we can conclude thanks to Theorem \ref{thm:kaw} that $\Exc(g)$ is uniruled.

	Now we prove that it is possible to choose $m > 0$ such that $mD$ is effective.
	Indeed, by the Hirzebruch--Riemann--Roch Theorem, 
	we have
	\begin{equation*}
	\chi(\mcO_X(mD)) = \deg{(\ch(mD)\cdot\td(\mathcal{T}_X))}_n.
	\end{equation*}
	Since $\ch(mD) = \sum_{k = 0}^\infty \frac{m^kD^k}{k!}$ because $mD$ is a line bundle, we obtain for $m$ sufficiently large
	\begin{equation*}
	\chi(\mcO_X(mD)) \sim \frac{m^nD^n}{n!} > 0.
	\end{equation*}
	By assumption,
	\begin{equation*}
	\chi(\mcO_X(mD)) = h^0(\mcO_X(mD))  - h^1(\mcO_X(mD)) \leq h^0(\mcO_X(mD)). 
	\end{equation*}
	Then $h^0(\mcO_X(mD)) > 0$ and $mD$ is effective. Since $mD$ is not nef and the pair $(X, \varepsilon mD)$ is klt for $\varepsilon$ sufficiently small, then we can conclude thanks to the logarithmic version of the Cone Theorem (see for instance \cite[Theorem 3.7]{KM98}).
	\end{proof}

We are going to apply Lemma \ref{lem:key} to $D_{a,b} \coloneq aH - bF$, where $H$ is an ample divisor on $X$ and $F$ is the generic fiber of a fibration of $X$ onto a curve. 
The point is to obtain an asymptotic vanishing of the higher cohomology of $D$ for some fixed values of $a$, $b$. Indeed, we are able to prove a slightly stronger statement, where we obtain a uniform vanishing for every $b$. 

\begin{lemma} \label{lem:van}
	Let $X$ be a projective manifold that admits a fibration onto a curve, let $H$ be an ample divisor on $X$ and let $F$ be the generic fiber of the fibration. Let $m_0 \in \mbN$ such that 
	\begin{itemize}
		\item $h^i(X, mH) = 0$, $i > 0$,
		\item $h^i(F, mH) = 0$, $i > 0$,
	\end{itemize}
	for all $m \geq m_0$. Then  $h^i(X, mH - kF) = 0$, $i > 1$, for all $m \geq m_0$ and for all $k \in \mbN$.
	
	In particular, for any positive integer $a$, $b$, the divisor $D_{a,b} = aH - bF$ satisfies the assumption of Lemma \ref{lem:key}. 
\end{lemma}

	It is always possible to choose $m_0$ satisfying the hypothesis of Lemma \ref{lem:van} thanks to Serre's vanishing.

\begin{proof}
	
	Given the standard exact sequence 
	\begin{equation*}
	0 \to \mcO_X(mH - kF) \to \mcO_X(mH) \to \mcO_{kF}(mH) \to 0,
	\end{equation*}
	it is enough to show that for every $i > 1$ we have
	\begin{equation*}
	h^{i - 1}(kF, \mcO_{kF}(mH)) = 0.
	\end{equation*}
	Indeed, let $D$ be an effective Cartier divisor on $X$ and 
	let $\mathcal{I}$ be the ideal sheaf of $D$ in $X$, i.e. $\mathcal{I} = \mathcal{O}_X(-D)$. 
	Then for any $k \in \mbN$, we have $\mathcal{O}_{(k+1)D} = \mathcal{O}_X/\mathcal{I}^{k+1}$ 
	and $\mathcal{O}_{kD} = \mathcal{O}_X/\mathcal{I}^k$. 
	Let $\mathcal{K}$ be defined by the short exact sequence
	\begin{equation*}
	0 \to \mathcal{K} \to \mcO_{(k + 1)D} \to \mcO_{kD} \to 0.
	\end{equation*}
	There is a natural sheaf isomorphism
	\begin{equation*}
	\mathcal{K} = \mathcal{I}^k/\mathcal{I}^{k+1}= \mathcal{I}^k \otimes_{\mathcal{O}_X} 
           \mathcal{O}_X/\mathcal{I} = \mathcal{O}_X(-kD)\otimes_{\mathcal{O}_X} \mathcal{O}_D =   
           \mathcal{O}_D(-kD).
	\end{equation*}
	Hence we obtain the exact sequence
	\begin{equation*}
	0 \to \mcO_D(-kD) \to \mcO_{(k + 1)D} \to \mcO_{kD} \to 0
	\end{equation*}
	and tensoring by $\mcO_X(mH)$ we get 
	\begin{equation*}
	0 \to \mcO_D(mH -kD) \to \mcO_{(k + 1)D}(mH) \to \mcO_{kD}(mH) \to 0.
	\end{equation*}
	Now, if we set $D =F$, since $F$ is a fiber we have  
	\begin{equation*}
	\mcO_F(mH -kF) = \mcO_F(mH),
	\end{equation*}
	so we can easily conclude by induction on $k$.
	
\end{proof}

\begin{remark} \label{rem:AndreottiGrauert}
	The required weaker vanishing condition of Lemma \ref{lem:key} for the divisor $D_{a,b}$ also follows from the next theorem, that can be seen as a consequence of the classical Andreotti--Grauert finiteness theorem.
\end{remark}
	\begin{theorem}[see {\cite[Chapter VII, Theorem 5.1]{Dem01}}] \label{thm:Demcurv}
		Let $X$ be a compact complex manifold of dimension $n$, let $s$ be a positive integer and $F$ be a hermitian line bundle such that its Chern curvature $i\Theta(F)$ has at least $n-s+1$ positive eigenvalues at every point of $X$. Then there exists a positive integer $l_0$ such that
		\begin{equation*}
		H^q(X, F^{\otimes l}) = 0, \text{ for all } l\ge l_0 \text{ and } q \ge s.
                     \end{equation*}
	\end{theorem}
	
	Indeed, let $\pi\colon X\to C$ be a fibration onto a curve and let $F$ be its generic fiber. Then as a divisor $F$ is the pull-back of a point $p\in C$. Call $A=\mathcal O_C(p)$ the corresponding ample line bundle, so that $\mathcal O_X(F)\simeq\pi^* A$. Put a metric $h_A$ on $A$ whose Chern curvature $i\Theta(A)$ is positive. Now take any ample divisor $H$ on $X$ and take a positively curved metric $h_H$ on $\mathcal O_X(H)$. We claim that for any integers $a, b>0$, the line bundle (associated to the divisor) $D_{a,b}=aH-bF$ has a metric whose Chern curvature has at least $n-1$ positive eigenvalues at every point of $X$. 
	Indeed, the metric $h_{a,b}=h_H^{\otimes a}\otimes\pi^*h_A^{\otimes -b}$ on $\mathcal O_X( D_{a,b})$ is such that
	$$
	\Theta\bigl(\mathcal O_X(D_{a,b})\bigr)=m\Theta\bigl(\mathcal O_X(H)\bigr)-k\pi^*\Theta(A).
	$$
	When one evaluates $i\Theta\bigl(\mathcal O_X(D_{a,b})\bigr)$ (seen as a hermitian form) on a non zero tangent vector $v$ on $X$ one gets then
	$$
	i\Theta\bigl(\mathcal O_X(D_{a,b})\bigr)(v)=m\,i\Theta\bigl(\mathcal O_X(H)\bigr)(v)-k\,i\Theta(A)(d\pi(v)).
	$$
	So, if $d\pi(v)=0$, then $i\Theta\bigl(\mathcal O_X(D_{a,b})\bigr)(v)>0$. It is then enough to observe that at every point $x\in X$, the kernel of $d\pi$ is at least of dimension $n-1$.

	Thus, we can apply Theorem \ref{thm:Demcurv} to deduce that 
	$$
	H^q(X,\mathcal O_X(\ell D_{a,b}))=0,\quad\text{for all $\ell\ge\ell_0$ and $q\ge 2$.}
	$$
	
%

We also need the following remark.
\begin{remark} \label{rmk:abel}
We first observe that $c_2(X)$ is non-zero, otherwise $X$ would be a finite unramified quotient of a torus and thus its fundamental group would contain a free abelian group of rank $2n$, contradicting our assumption of simple connectedness. 

Indeed, by Yau's celebrated solution \cite{Yau78} of the Calabi conjecture, $X$ admits a Ricci-flat K\"ahler metric (unique in each K\"ahler class), since $c_1(X)=0$. Therefore, if $c_2(X)$ were zero, then by the equality case in the Kobayashi--L\"ubke inequality $X$ would be flat, so by the classical theorem of Bieberbach $X$ would be covered by a complex torus (see \cite[Chapter IV, (4.15) Corollary]{Kob87}).

Next, given the fibration $f \colon X \to C$, thanks to the relative tangent sequence on a generic fiber $F$, we get 
%
	\begin{equation*}
	0 \to T_F \to T_{X|_{F}} \to \mcO_F \to 0.
	\end{equation*}
Hence we deduce
\begin{equation*}
c_2(X) \cdot F = c_2(T_{X|_{F}}) = c_2(T_F) = c_2(F).
\end{equation*}
Since $F$ is an abelian variety, we have $c_2(X) \cdot F = c_2(F) = 0$ as a cycle.
\end{remark}

We now give the proof of Theorem \ref{thm:applic}.
Let $F$ be the generic fiber of the fibration and let $H$ be an ample divisor on $X$. We consider the affine line of divisors (with rational slope) $N_t = H - tF$ for $t \in \mbQ$. If we let 
$n = \dim X$ and
\begin{equation*}
t_0 = \frac{H^n}{nH^{n-1}\cdot F} \in \mbQ,
\end{equation*}
then  we have $(N_t)^n > 0$ for each $t < t_0$ and $(N_{t_0})^n=0$.
Now there are two possible cases:
\begin{itemize}
	\item [(I)] $N_t$ is nef for each $t < t_0$.
	\item [(II)] There exists a $\bar{t} < t_0$ such that $N_{\bar{t}}$ is non-nef.
\end{itemize}

Let us focus on case (I) first. Since being nef is a closed condition, $N_{t_0}$ is also nef.
Next, it is easy to verify that $(N_{t_0})^{n-1}\cdot H = (1 - \frac{n-1}{n})H^n > 0$. 
Finally, according to Remark \ref{rmk:abel}, $c_2(X) \cdot F = 0$, hence we can conclude that $c_2(X)\cdot (N_{t_0})^{n - 2} =  c_2(X)\cdot H^{n - 2} > 0$ thanks to \cite[Theorem 1.1]{Miy87}, because $H$ is ample and $c_2(X) \ne 0$, see Remark \ref{rmk:abel}. 
Since we have $\nu(N_{t_0}) = n - 1$ we can conclude thanks to Corollary \ref{cor:num}.

Let us consider now case (II). Let $a$ and $b$ positive natural numbers such that $\bar{t} = \frac{b}{a}$. Then we have that $aN_{\bar{t}} = aH - bF = D_{a,b}$ is as in Lemma \ref{lem:van} and since $(aN_{\bar{t}})^n = a^n(N_{\bar{t}})^n > 0$ we can conclude thanks to Lemma \ref{lem:key}. 

\medskip

Finally, we give a proof of Corollaries \ref{cor:multsect} and \ref{cor:kob}.

\begin{proof} [Proof of Corollaries \ref{cor:multsect} and \ref{cor:kob}]
Let $F$ be the fiber of the abelian fibration, and let examine the two possible cases that occur during the proof of Theorem \ref{thm:applic}. 

In case (I), we get a a semiample divisor $N_{t_0}$ which induces an elliptic fiber space structure on $X$. Now, following the proof of Theorem \ref{thm:main}, either this fiber space is equidimensional or not. If it is equidimensional, then we necessarily have some singular fiber and the rational curve is found therein. It can be either a (singular) irreducible rational curve, or it can be a chain of smooth rational $(-2)$-curves. The class of such a fiber is given by $N_{t_0}^{n-1}$, where $n = \dim X$, and
$$
F\cdot N_{t_0}^{n-1}= \frac{1}{t_0}(H - N_{t_0})\cdot N_{t_0}^{n-1}= \frac{1}{t_0} H\cdot N_{t_0}^{n-1} - \frac{1}{t_0} N_{t_0}^n =
 \frac{1}{t_0} H\cdot N_{t_0}^{n-1}> 0.
$$
This means that this fiber is transversal to the abelian fibration, which provides immediately a rational multi-section if the fiber is irreducible. If it is a chain of rational curve, it suffices to observe that at least one of its components has to intersect positively $F$, since the whole fiber is an effective cycle intersecting $F$ positively.

Finally, if we are not in the equidimensional case, this means that there is a uniruled exceptional locus and inside it we find a rational curve, say $R$, which is contracted. But then $N_{t_0}\cdot R=0$, so
$$
F\cdot R=\frac{1}{t_0}(H - N_{t_0})\cdot R=\frac{1}{t_0}\,H\cdot R>0,
$$ 
and therefore $R$ is again transversal to the original fibration.

If, on the other hand, we are in case (II), the existence of a rational curve $R$ on $X$ is deduced from the logarithmic version of the Cone 
Theorem, in particular we have $N_{\bar{t}}\cdot R < 0$. Hence,
$$
F\cdot R = \frac{1}{\bar{t}}(H - N_{\bar{t}})\cdot R = \frac{1}{\bar{t}} H\cdot R- \frac{1}{\bar{t}}N_{\bar{t}}\cdot R > 0,
$$
and thus also in this case we get that $R$ is transverse to the starting abelian fibration. 

Summing up, we have found in all possible cases the desired rational curve giving a multi-section of the abelian fibration, and this proves Corollary \ref{cor:multsect}.

Finally, let $p,q\in X$ be two distinct points. If they both lie in a general fiber of the abelian fibration, which is a complex torus, then $d_X(p,q)=0$. If they both lie in a special fiber, their Kobayashi distance is again zero, since any special fiber is the limit of general fibers and the Kobayashi pseudodistance $d_X\colon X\times X\to\mathbb R$ is continuous in the Euclidean topology (see \cite[(3.1.13) Proposition]{Kob98}). If $p$ and $q$ lie in different fibers, we use the transversal rational curve found above to move from one fiber to another keeping zero Kobayashi distance, and we are done.
\end{proof}

\section{Appendix: A sufficient numerical criterion by Koll\'ar}\label{sec:kol}

To finish with, since we used it in a prominent way all along the paper, for the sake of completeness we give an overview of Koll\'ar's proof of the abundance-type result for nef line bundles of numerical dimension $n-1$ \cite[Corollary 11]{Kol15}, in the slightly less general setting we need here (cf. Corollary \ref{cor:num}). So, let $X$ be a smooth $n$-dimensional projective manifold with trivial canonical bundle $K_X\simeq\mathcal O_X$, and suppose that $L\to X$ is a nef holomorphic line bundle of numerical dimension $\nu(L)=n-1$. One wants to show that $L$ is semi-ample. For this purpose, it is sufficient to show that the Kodaira--Iitaka dimension $\kappa(L)$ of $L$, which is in general less than or equal to its numerical dimension, satisfies indeed $\kappa(L)=\nu(L)$. For, in this case one can then apply the following theorem, which we state in an oversimplified version.

\begin{theorem}[see {\cite{Kaw85,KMM87} and also \cite[Theorem 1.1]{Fuj11}}]
	Let $X$ be a smooth projective manifold with trivial canonical bundle, and $L\to X$ a nef line bundle. Then $L$ is semi-ample if and only if its Kodaira--Iitaka dimension equals its numerical dimension.
\end{theorem}

Observe that, in the theorem above, the \lq\lq only if\rq\rq{} part is straightforward, and the importance of this statement really relies upon the \lq\lq if\rq\rq{} part. Now, the asymptotic Riemann--Roch formula for $L$ reads, under our assumptions:
$$
\chi(X,L^{\otimes m})=\frac{L^{n-2}\cdot c_2(X)}{12(n-2)!}\,m^{n-2}+O(m^{n-3}).
$$
On the other hand, the Kawamata-Viehweg vanishing theorem \cite[Special case 6.13]{Dem01} gives, since $\nu(L)=n-1$ and $K_X\simeq\mathcal O_X$, that $H^q(X,L^{\otimes m})=\{0\}$ for all $m\ge 1$ and all $q>1$. Therefore, we obtain
$$
\begin{aligned}
h^0(X,L^{\otimes m}) & \ge h^0(X,L^{\otimes m}) - h^1(X,L^{\otimes m}) \\
& = \chi(X,L^{\otimes m})=\frac{L^{n-2}\cdot c_2(X)}{12(n-2)!}\,m^{n-2}+O(m^{n-3}).
\end{aligned}
$$
Thus, if $L^{n-2}\cdot c_2(X)>0$, we get that $\kappa(L)\ge n-2$. By \cite{Miy87}, it always holds true that $L^{n-2}\cdot c_2(X)\ge 0$, and therefore if $L^{n-2}\cdot c_2(X)\ne 0$, then indeed we have that $\kappa(L)\ge n-2$. 

We now proceed by contradiction and suppose that $\kappa(L)=n-2<n-1=\nu(L)$. Let's consider the Iitaka fibration (cf. \cite[Theorem 2.1.33]{Laz04}) associated to $L$. Namely, there exists for any large $m$ divisible enough a commutative diagram 
$$
\xymatrix{
	X \ar@{-->}[d]_{\phi_m} & X_\infty \ar[l]_{u_\infty} \ar[d]^{\phi_\infty}\\
	Y_m & Y_\infty \ar@{-->}[l]^{\nu_m}
}
$$
of rational maps and morphisms, where the horizontal maps are birational, $u_\infty$ is a morphism and $\phi_\infty$ determines an algebraic fiber space structure on $X_\infty$. Here, $\phi_m$ is the rational map associated to the linear system $|L^{\otimes m}|$. One has that $\dim Y_\infty=\kappa(X,L)=n-2$ and moreover, if we set $L_\infty=u_\infty^*L$ and take $S\subset X_\infty$ to be a very general fiber of $\phi_\infty$ (so that $S$ is in particular is a smooth surface), then $\kappa(S,L_\infty|_S)=0$. Observe that one can suppose $X_\infty$ to be smooth and hence $Y_\infty$ normal. 

Next, remark that, since $(u_\infty)_*\mathcal O_{X_\infty}\simeq\mathcal O_X$, one has that $\kappa(L_\infty)=\kappa(L)=n-2$. Moreover, since $u_\infty$ is a proper surjective map and $X_\infty$ is projective, $L_\infty$ is nef and $\nu(L_\infty)=\nu(L)=n-1$. 
Another remark is that $K_{X_\infty}$ is linearly equivalent to the effective divisor given by $E:=\{\det du_\infty=0\}$. 

\begin{lemma}\label{lem:kod0}
	The surface $S$ has Kodaira dimension zero.
\end{lemma}

\begin{proof}

The short exact sequence induced by the differential of $\phi_\infty$ gives $K_S\simeq K_{X_\infty}|_S$. Since $K_{X_\infty}\sim E\ge 0$ is effective and $S$ is not contained in $E$, it follows that $\kappa(S)\ge 0$. We now claim that $\kappa(S,K_S\otimes L_\infty|_S)=0$. This is enough to conclude. Indeed, since (some power of) $L_\infty|_S$ is effective, then 
$$
\kappa(S)=\kappa(S,K_S)\le\kappa(S,K_S\otimes L_\infty|_S)=0.
$$
It follows that $\kappa(S)=0$.

We now come back to the claim. Observe that, by the projection formula, $(u_\infty)_*\mathcal O_{X_\infty}(K_{X_\infty}^{\otimes m}\otimes L_\infty^{\otimes m})\simeq\mathcal O_X(L^{\otimes m})$, since we have that $\mathcal O_X(K_X)\simeq\mathcal O_X$ and thus $(u_\infty)_*\mathcal O_{X_\infty}(K_{X_\infty}^{\otimes m})\simeq\mathcal O_X$. This implies that the map $\psi_m$ induced by the linear system $|K_{X_\infty}^{\otimes m}\otimes L_\infty^{\otimes m}|$ is nothing but the composite $\phi_m\circ u_\infty$. To conclude, it suffices to follow \textsl{verbatim} \cite[Proof of Theorem 2.1.33, Step 3]{Laz04}, replacing $L_\infty$ by $K_{X_\infty}\otimes L_\infty$ (observe that, by \cite[Remark 2.1.34]{Laz04}, a posteriori we get that $\phi_\infty$ is then also the Iitaka fibration associated to $K_{X_\infty}\otimes L_\infty$).
\end{proof}

Now, take $D$ to be an effective divisor in the linear system $|L_\infty^{\otimes m}|$, $m\gg 1$. Then either $D$ is $\phi_\infty$-horizontal, \textsl{i.e.} $\phi_\infty(D)=Y_\infty$, or it is $\phi_\infty$-vertical, \textsl{i.e.} $\phi_\infty(D)\subsetneq Y_\infty$. In the first case, we shall obtain a contradiction by showing that this would imply $\kappa(S,L_\infty|_S)\ge 1$; in the second case, we shall get that $\nu(L_\infty)<n-1$, which is again a contradiction.

\subsection*{The horizontal case} In this case, we are in the situation where the restriction $D_S$ of $D$ to $S$ is a nef, effective, non zero divisor. In particular $D_S$ is not numerically trivial. Let $S_{\text{min}}$ be the minimal model of $S$, obtained by contracting all $(-1)$-curves on $S$. Call $\mu\colon S\to S_{\text{min}}$ the corresponding morphism. Observe that $D_S$ cannot be entirely contracted by $\mu$, since in this case one would have $D_S^2<0$, contradicting the nefness of $D_S$. Indeed, take any ample Cartier divisor $H$ on $S_{\text{min}}$. Its pull-back $\mu^*H$ to $S$ is big and nef so that $\mu^*H^2>0$; on the other hand, if $D_S$ were entirely contracted by $\mu$, we would have $\mu^*H\cdot D_S=0$ since $\mu^*H|_{D_S}$ would be trivial. Since $D_S$ is not numerically trivial then the Hodge Index Theorem would give, as claimed, $D_S^2<0$. Thus, the proper push-forward $\mu_*D_S=:D_{S_{\text{min}}}$ is again a nef effective non zero divisor on $S_{\text{min}}$. Observe that $\kappa(D_S)\ge\kappa(D_{S_{\text{min}}})$. Now, since $\kappa(S_\text{min})=\kappa(S)$, $\kappa(S)=0$ by Lemma \ref{lem:kod0} and $K_{S_\text{min}}$ is nef, then by abundance for surfaces we see that $K_{S_\text{min}}$ is semi-ample, and thus a torsion line bundle. In particular, $\kappa(D_{S_{\text{min}}})=\kappa(K_{S_{\text{min}}}+D_{S_{\text{min}}})$. Now, there exists a (small) positive rational number $\varepsilon$ such that $(S_\text{min},\varepsilon D_{S_{\text{min}}})$ is a klt pair, and moreover $K_{S_{\text{min}}}+\varepsilon D_{S_{\text{min}}}$ is nef. Thus, by log-abundance for surfaces $K_{S_{\text{min}}}+\varepsilon D_{S_{\text{min}}}$ is semi-ample and it is not numerically trivial by construction. This implies that $\kappa(K_{S_{\text{min}}}+\varepsilon D_{S_{\text{min}}})\ge 1$. But 
$$
\kappa(D_S)\ge\kappa(D_{S_{\text{min}}})=\kappa(\varepsilon D_{S_{\text{min}}})=\kappa(K_{S_{\text{min}}}+\varepsilon D_{S_{\text{min}}})\ge 1,
$$
contradiction.

\subsection*{The vertical case} We finally treat the case where $\phi_\infty(D)\subsetneq Y_\infty$. The hypothesis implies that the coherent sheaf $(\phi_\infty)_*\mathcal O_{X_\infty}(-D)$ on $Y_\infty$ is non zero. Thus, by the Cartan--Serre--Grothendieck Theorem, there exists an ample line bundle $A\to Y_\infty$ such that $H^0\bigl(Y_\infty,A\otimes(\phi_\infty)_*\mathcal O_{X_\infty}(-D)\bigr)\ne\{0\}$. By the projection formula we thus get $H^0\bigl(X_\infty,\phi_\infty^*A\otimes\mathcal O_{X_\infty}(-D)\bigr)\ne\{0\}$. In particular, $\phi_\infty^*A$ is linearly equivalent to $D+F$, where $F$ is an effective divisor on $X_\infty$. 

Therefore, for all integer $r\ge 1$, we have
$$
\begin{aligned}
\phi_\infty^*A^r-D^r & =(\phi_\infty^*A-D)\cdot\biggl(\sum_{i=0}^{r-1}\phi_\infty^*A^{i}\cdot D^{r-1-i}\biggr) \\
& = F\cdot\biggl(\sum_{i=0}^{r-1}\phi_\infty^*A^{i}\cdot D^{r-1-i}\biggr)\ge 0,
\end{aligned}
$$
since both $\phi_\infty^*A$ and $D$ are nef and $F$ is effective. In particular, 
$$
0\le D^{n-1}\le\phi_\infty^*A^{n-1}=0,
$$
since $\phi_\infty^*A$ comes from an $(n-2)$-dimensional variety. Thus, we have that $\nu(L_\infty)=\nu(D)<n-1$, which gives the desired contradiction.

\end{document}